\documentclass[a4paper,reqno]{amsart}

\usepackage{amssymb}
\usepackage{amsmath}
\usepackage{amsthm}
\usepackage{bbm}
\usepackage{tikz}

      \def\cF{{\mathcal F}}
      
      \def\cL{{\mathcal L}}
      
\def\cP{{\mathcal P}}   \def\cQ{{\mathcal Q}}

\def\cal H{{\mathcal H}}

\def\R{\mathbb{R}}
\def\C{\mathbb{C}}

\def\N{\mathbb{N}}

\def\ran{{\text{\rm ran\,}}}
\def\dom{{\text{\rm dom\,}}}

\def\phi{\varphi}

\DeclareMathOperator{\spann}{span}

\renewcommand{\theta}{\vartheta}

\newtheorem{theorem}{Theorem}[section]
\newtheorem*{thm*}{Theorem}
\newtheorem{proposition}[theorem]{Proposition}

\newtheorem{lemma}[theorem]{Lemma}

\theoremstyle{definition}
\newtheorem{definition}[theorem]{Definition}
\newtheorem{hypothesis}[theorem]{Hypothesis}
\newtheorem{classification}[theorem]{Classification}
\newtheorem{example}[theorem]{Example}

\newtheorem{remark}[theorem]{Remark}

\numberwithin{equation}{section}

\title{Spectral monotonicity for Schr\"odinger operators on metric graphs}

\author[J.~Rohleder]{Jonathan Rohleder}
\address{Matematiska institutionen\\ Stockholms universitet \\
106 91 Stockholm \\
Sweden}
\email{jonathan.rohleder@math.su.se}

\author[C.~Seifert]{Christian Seifert}
\address{TU Hamburg \\ Institut f\"ur Mathematik \\
Am Schwarzenberg-Campus~3 \\
Geb\"aude E \\
21073 Hamburg \\
Germany}
\email{christian.seifert@tuhh.de}

\begin{document}



\begin{abstract}
We study the influence of certain geometric perturbations on the spectra of self-adjoint Schr\"odinger operators on compact metric graphs. Results are obtained for permutation invariant vertex conditions, which, amongst others, include $\delta$ and $\delta'$-type conditions. We show that adding edges to the graph or joining vertices changes the eigenvalues monotonically. However, the monotonicity properties may differ from what is known for the previously studied cases of Kirchhoff (or standard) and $\delta$-conditions and may depend on the signs of the coefficients in the vertex conditions. 
\end{abstract}

\maketitle

\section{Introduction}

Differential operators on metric graphs, also called quantum graphs, have been studied extensively during the last two decades, see e.g.\ the two monographs \cite{BK13,Mugnolo2014} and the references therein. An aspect which came into focus rather recently is the behavior of the spectrum of Laplacians (or more general Schr\"odinger operators) on metric graphs under perturbations of the geometry and topology, see, e.g,~\cite{EJ12,KMN13}. These so-called surgery principles appear to be a natural and very useful tool for spectral investigation. They were applied successfully to derive Faber--Krahn type inequalities for graph Laplacians and eigenvalue estimates depending on various quantities of the graph such as the total length, the number of edges or vertices, the diameter or the Betti number, see~\cite{A16,BL17,BKKM17,KKT16,K18,KKMM16,KN17,KN14,R17}, where Kirchhoff (also called standard), $\delta$, or Dirichlet conditions at the vertices were treated.

In the present paper two types of perturbations of the graph are considered for a more general class of vertex conditions. More specifically, for a compact metric graph $\Gamma$ we focus on the change of the eigenvalues of a self-adjoint Schr\"odinger operator $H$ in $L^2 (\Gamma)$ when either
\begin{itemize}
 \item[(i)] an edge is added to the graph, or
 \item[(ii)] two vertices of the graph are joined into a single vertex.
\end{itemize}
For the case of Kirchhoff or $\delta$ vertex conditions it is known that the eigenvalues behave non-increasing under the perturbation (i) if the additional edge connects a vertex of $\Gamma$ to a new vertex of degree one while no general monotonicity principle is valid if the new edge connects two previously existing vertices, see~\cite{KMN13}. For the graph transformation (ii), the eigenvalues of Schr\"odinger operators subject to Kirchhoff or $\delta$-conditions are known to move non-decreasingly. 

In order to study eigenvalue monotonicity under the graph transformations (i) and (ii) for more general couplings, one needs to specify how vertex conditions change; actually, each of these transformations leads to an increase of the vertex degree of certain vertices. Hence, we will consider the question only for those vertex conditions which admit a canonical or {\em natural} extension to a larger vertex degree. This is the case if the vertex conditions are permutation invariant, i.e., different edges incident to the same vertex are not distinguished by the vertex conditions. Permutation invariant vertex conditions include $\delta$- and $\delta'$-type conditions. They are discussed in detail and classified in Section~\ref{sec:SO} below; cf.\ Classifiction~\ref{class}.

In the main results of this paper we observe that different types of permutation invariant vertex conditions behave differently under the considered transformations. Section~\ref{sec:attaching_edges} is devoted to the transformation (i). It turns out that there is a class of conditions for which, in contrast to Kirchhoff or $\delta$-conditions, all eigenvalues behave monotonically non-increasing if an edge is added connecting two previously present vertices; this class includes so-called anti-Kirchhoff as well as $\delta'$-type conditions, see Theorem~\ref{thm:monotonicity}. Moreover, in Theorem~\ref{thm:monotonicity2} it is shown that for all permutation invariant vertex conditions the eigenvalues behave non-increasingly if an edge connecting the specified vertex to a new vertex of degree one is added. These observations are complemented by examples. In Section~\ref{sec:joining_vertices} the behavior of the eigenvalues under the transformation (ii) is studied. It turns out that this transformation divides the permutation invariant conditions into three classes: those for which joining two vertices leads to non-decreasing eigenvalues (as for Kirchhoff and $\delta$-conditions), those for which it leads, conversely, to non-increasing eigenvalues, and those for which the monotonicity properties depend on the signs of the coefficients in the vertex conditions. The latter applies e.g.\ to $\delta'$-type conditions. The different classes of permutation invariant vertex conditions according to Classification \ref{class} are treated in Theorems \ref{thm:joining}, \ref{thm:joining_delta'} and \ref{thm:joining_ExnerConditions}.
We would like to mention that all results depend only on the vertex conditions at those vertices which are changed by the graph transformation. At all other vertices, general self-adjoint conditions are allowed.

The proofs of our results are all variational comparing Rayleigh quotients, which is a standard method in obtaining eigenvalue estimates. In fact, estimates on the quadratic form asociated to the Schr\"odinger operator on suitable finite-dimensional subspaces together with an application of the min-max principle yields the desired estimates for the eigenvalues.

After conceiving the paper we learned about the manuscript \cite{BKKM18+} dealing also with monotonicity properties for the spectrum
of the Laplacian with Kirchhoff, $\delta$, and Dirichlet boundary conditions, but for a larger toolkit of surgery principles. 

\section{Schr\"odinger operators with permutation invariant vertex conditions}
\label{sec:SO}

In this section we introduce the operators under consideration. First we recall some general facts on self-adjoint vertex (or coupling) conditions. After that we restrict our considerations to a subclass, the permutation invariant conditions, which is suitable for the questions under investigation.

Let $\Gamma$ be a finite, compact metric graph, i.e.\ a graph consisting of a finite vertex set $V:=V(\Gamma)$ and a finite edge set $E:=E(\Gamma)$ that is, additionally, equipped with a length function $L : E \to (0, \infty)$. We identify each edge $e\in E$ with the interval $[0, L(e)]\subseteq \R$ and obtain a natural metric on $\Gamma$. For each $e \in E$ we say that $e$ has initial vertex $v_{\rm i}$ and terminal vertex $v_{\rm t}$ if $e$ is incident to $v_{\rm i}$ and $v_{\rm t}$ such that $v_{\rm i}$ is identified with the zero endpoint of $[0, L (e)]$ and $v_{\rm t}$ is identified with the endpoint $L (e)$; note that $v_{\rm i}$ and $v_{\rm t}$ coincide if $e$ is a loop. Furthermore, for each vertex $v\in V$ we denote by $E_{v, \rm i} \subseteq E$ ($E_{v, \rm t} \subseteq E$) the set of edges for which $v$ is the initial vertex (terminal vertex) and by $\deg (v) = |E_{v, \rm i}| + |E_{v, \rm t}|$ the degree of $v$. Finally, by $L^2(\Gamma)$ we denote the usual $L^2$-space on $\Gamma$, which coincides with the direct sum $\bigoplus_{e\in E} L^2(0, L(e))$. For $f \in L^2 (\Gamma)$ we denote by $f_e$ the restriction of $f$ to some edge $e \in E$. Moreover, for $k = 1, 2, \dots$ we make use of the Sobolev spaces
\begin{align*}
 \widetilde H^k (\Gamma) := \bigoplus_{e \in E} H^k (0, L (e)),
\end{align*}
equipped with the standard Sobolev norms and inner products. Moreover, a function $f \in \widetilde H^1 (\Gamma)$ is called continuous at a vertex $v \in V$ if for any two edges $e, \hat e \in E$ incident to $v$ the values of $f_e$ and $f_{\hat e}$ at $v$ coincide.

In the following we consider Schr\"odinger operators in $L^2 (\Gamma)$ acting as 
\begin{align}\label{eq:Schroedinger}
 (\cL f)_e = - f_e'' + q_e f_e, \quad e \in E,
\end{align}
with a real-valued potential $q$; for simplicity, we assume that $q \in L^\infty(\Gamma)$. Sometimes we will impose mild sign conditions on $q$; cf.\ Remark \ref{rem:Potentials}.

In order to write down vertex conditions we make use of the following abbreviations. 
For any vertex $v \in V$ we fix enumerations $\{e_1, \dots, e_l\}$ and $\{e_{l + 1}, \dots e_{m} \}$ of $E_{v, \rm i}$ and $E_{v, \rm t}$, respectively, where $l = |E_{v, \rm i}|$ and $m = \deg (v)$. For each sufficiently regular $f \in L^2 (\Gamma)$ we write
\begin{align*}
 F (v) := \begin{pmatrix} f_{e_1} (0) \\ \vdots \\ f_{e_l} (0) \\ f_{e_{l + 1}} (L (e_{l + 1})) \\ \vdots \\ f_{e_{m}} (L (e_{m})) \end{pmatrix} \quad \text{and} \quad F' (v) := \begin{pmatrix} f'_{e_1} (0) \\ \vdots \\ f'_{e_l} (0) \\ - f'_{e_{l + 1}} (L (e_{l + 1})) \\ \vdots \\ - f'_{e_{m}} (L (e_{m})) \end{pmatrix}.
\end{align*}
Note that $F (v)$ is well-defined whenever $f \in \widetilde H^1 (\Gamma)$ and $F' (v)$ is well-defined for $f \in \widetilde H^2 (\Gamma)$. The latter denotes the collection of derivatives pointing out of $v$ into the edges. The following description of all self-adjoint incarnations of  $\cL$ in $L^2 (\Gamma)$ with local vertex conditions is standard, see, e.g.,~\cite[Theorem~1.4.4]{BK13}.

\begin{proposition}\label{prop:SAconditions}
Let $\Gamma$ be a finite, compact metric graph, let $q \in L^\infty (\Gamma)$ be real-valued and let $\cL$ be the Schr\"odinger differential expression in~\eqref{eq:Schroedinger}. For each vertex $v \in V$ let $P_{v, \rm D}, P_{v, \rm N}$ and $P_{v, \rm R}$ be orthogonal projections in $\C^{\deg (v)}$ with mutually orthogonal ranges such that $P_{v, \rm D} + P_{v, \rm N} + P_{v, \rm R} = I$ and let $\Lambda_v$ be a self-adjoint, invertible operator in $\ran P_{v, \rm R}$. Then the operator $H$ in $L^2 (\Gamma)$ given by
\begin{align*}
 H f & = \cL f, \\
 \dom H & = \Big\{ f \in \widetilde H^2 (\Gamma) : P_{v, \rm D} F (v) = 0, P_{v, \rm N} F' (v) = 0, \\
 & \qquad \qquad P_{v, \rm R} F' (v) = \Lambda_v P_{v, \rm R} F (v)~\text{for each}~v \in V \Big\},
\end{align*}
is self-adjoint (and each self-adjoint realization of $\cL$ in $L^2 (\Gamma)$ subject to local vertex conditions can be written in this form). Furthermore, the closed quadratic form $h$ corresponding to the operator $H$ is given by
\begin{align*}
 h [f] & = \int_\Gamma |f'|^2 \textup{d}x + \int_\Gamma q |f|^2 \textup{d}x + \sum_{v \in V} \big\langle \Lambda_v P_{v, \rm R} F (v), P_{v, \rm R} F (v) \big\rangle, \\
 \dom h & = \left\{ f \in \widetilde H^1 (\Gamma) : P_{v, \rm D} F (v) = 0~\text{for each}~v \in V \right\}.
\end{align*}
\end{proposition}

Recall that by a standard compact embedding argument the spectrum of the Hamiltonian $H$ on the compact graph $\Gamma$ is always purely discrete and bounded from below, see, e.g.,~\cite[Corollary 10 and Theorem 18]{Kuchment2004}. In the following we denote by
\begin{align*}
 \lambda_1 (H) \leq \lambda_2 (H) \leq \dots 
\end{align*}
the eigenvalues of $H$ in non-decreasing order, counted with multiplicities. We remark that all eigenvalues are non-negative if $q\geq 0$ and $\Lambda_v$ is non-negative for each vertex $v$. If $H$ acts as the Laplacian, i.e.\ $q = 0$ identically on $\Gamma$, and the vertex conditions are Kirchhoff conditions at every vertex then $\lambda_1 (H) = 0$ and the multiplicity of $\lambda_1 (H)$ coincides with the number of connected components of $\Gamma$.

The focus of this paper is on the behavior of the spectrum under a change of the graph, namely adding extra edges to a vertex or joining two vertices. For general vertex conditions there is no natural extension in the case that edges are added to a vertex. However, such an extension exists if the vertex conditions are {\em permutation invariant}, that is, they do not distinguish between the edges incident to the vertex. However, the only subspaces of $\C^{\deg (v)}$ being invariant under all permutations are $\{0\}$, $\C^{\deg (v)}$, $\spann \{ (1, 1, \dots, 1)^\top \}$ and $(\spann \{ (1, 1, \dots, 1)^\top \})^\perp$. Therefore in the following we assume that each of the orthogonal projections $P_{v, \rm D}$, $P_{v, \rm N}$ and $P_{v, \rm R}$ involved in the vertex conditions projects onto one of these subspaces. The orthogonal projections onto $\spann \{ (1, 1, \dots, 1)^\top \}$ and $(\spann \{ (1, 1, \dots, 1)^\top \})^\perp$ are given by the $d \times d$-matrices
\begin{align*}
 \cP:=\cP_d = \begin{pmatrix} \frac{1}{d} & \dots & \frac{1}{d} \\ \vdots & & \vdots \\ \frac{1}{d} & \dots & \frac{1}{d} \end{pmatrix} \quad \text{and} \quad \cQ:=\cQ_d = \begin{pmatrix} \frac{d - 1}{d} & - \frac{1}{d} & \dots & - \frac{1}{d} \\ - \frac{1}{d} & \ddots & \ddots & \vdots \\ \vdots & \ddots & \ddots & - \frac{1}{d} \\ - \frac{1}{d} & \dots & - \frac{1}{d} & \frac{d - 1}{d} \end{pmatrix},
\end{align*}
where $d := \deg (v)$. Moreover, in order to make the conditions permutation invariant we assume that $\Lambda_v$ is the multiplication by a constant. Note that, under these assumptions, if one of the projections is the identity (and, thus, the others are zero) then the vertex conditions do not reflect the connectivity of the graph but degenerate to decoupled Dirichlet, Neumann or Robin conditions. As we are not interested in this situation, we are left with the following assumption for the vertices to be changed. It comprises all permutation invariant vertex conditions that do not decouple the vertex.

\begin{hypothesis}\label{hyp:PQ}
For a given vertex $v \in V$ we assume that $P_{v, \rm D}, P_{v, \rm N}, P_{v, \rm R} \in \{0, \cP, \cQ\}$ such that $P_{v, \rm D}, P_{v, \rm N}$ and $P_{v, \rm R}$ are mutually distinct. Moreover, we assume that $\Lambda_v$ is the operator of multiplication by a constant in $\ran P_{v, \rm R}$.
\end{hypothesis}

Considering the vertex conditions in Proposition~\ref{prop:SAconditions} under the additional assumption of Hypothesis~\ref{hyp:PQ} leads to a total of six different classes of conditions. They are described in the following.

\begin{classification}\label{class}
Let Hypothesis~\ref{hyp:PQ} hold for some vertex $v \in V$. Then the vertex conditions at $v$ have one of the following forms.
\begin{enumerate}
	\item[I.] The first two cases correspond to $P_{v, \rm D} = \cQ$.
	\begin{enumerate}
	 \item[(a)] If $P_{v, \rm N} = \cP$ and $P_{v, \rm R} = 0$ we get {\em Kirchhoff conditions}
	\begin{align*}
	 f~\text{is continuous at}~v \quad \text{and} \quad \sum_{j = 1}^{\deg (v)} F_j' (v) = 0.
	\end{align*}
	\item[(b)]	For $P_{v, \rm N} = 0$, $P_{v, \rm R} = \cP$ and $\Lambda_v$ acting as multiplication by $\frac{\alpha_v}{\deg (v)}$ for some real $\alpha_v \neq 0$ we have {\em $\delta$-conditions}
	\begin{align*}
	 f~\text{is continuous at}~v \quad \text{and} \quad \sum_{j = 1}^{\deg (v)} F_j' (v) = \alpha_v f (v).
	\end{align*}
	\end{enumerate}
	
	\item[II.] The next two cases are $P_{v, \rm N} = \cQ$. 
	\begin{enumerate}
	 \item If $P_{v, \rm D} = \cP$ and $\cP_{v, \rm R} = 0$ we have conditions which are known as {\em anti-Kirchhoff}, namely, the vector $F' (v)$ is constant and
	\begin{align*}
	 \sum_{j = 1}^{\deg (v)} F_j (v) = 0.
	\end{align*}
	 \item If $P_{v, \rm D} = 0$, $P_{v, \rm R} = \cP$ and $\Lambda_v$ is multiplication by $\frac{\deg (v)}{\beta_v}$ for some real $\beta_v \neq 0$ we get {\em $\delta'$-type conditions}, i.e., the vector $F' (v)$ is constant (let $f' (v)$ denote an arbitrary component of it) and
	\begin{align*}
	 \sum_{j = 1}^{\deg (v)} F_j (v) = \beta_v f' (v).
	\end{align*}
	\end{enumerate}
	
	\item[III.] The remaining cases correspond to $P_{v, \rm R} = \cQ$.
	\begin{enumerate}
		\item[(a)] For $P_{v, \rm D} = \cP$, $P_{v, \rm N} = 0$ and $\Lambda_v$ being multiplication with real $C_v \neq 0$ the conditions can be written as
	\begin{align*}
	 \sum_{j = 1}^{\deg (v)} F_j (v) = 0 \quad \text{and} \quad F_j' (v) - F_k' (v) = C_v (F_j (v) - F_k (v))
	\end{align*}
	for all $j, k \in \{1, \dots, \deg (v)\}$.
	\item[(b)] If $P_{v, \rm D} = 0$, $P_{v, \rm N} = \cP$ and $\Lambda_v$ is multiplication by $\frac{1}{D_v}$ for a real $D_v \neq 0$ we get
	\begin{align*}
	 \sum_{j = 1}^{\deg (v)} F_j' (v) = 0 \quad \text{and} \quad F_j (v) - F_k (v) = D_v (F_j' (v) - F_k' (v))
	\end{align*}
	for all $j, k \in \{1, \dots, \deg (v)\}$.
	\end{enumerate}
\end{enumerate}
\end{classification}

\begin{remark}
The conditions of type III in Classification~\ref{class} appear less frequently in the literature. However, for instance the conditions III~(b) were proposed in~\cite{E96} as an alternative to the $\delta'$-type conditions II~(b) in the description of quantum particles on graphs, and their physical properties were discussed there.
\end{remark}

As all proofs in the following sections will be based on calculations involving quadratic forms, we provide the following lemma. Its proof is a simple calculation and is left to the reader.

\begin{lemma}\label{lem:forms}
Let $\Gamma$ be a finite, compact metric graph, let $H$ be a self-adjoint Schr\"odinger operator in $L^2 (\Gamma)$ as in Proposition~\ref{prop:SAconditions}, and let $h$ be the corresponding quadratic form. Furthermore, let $v \in V$ and let the vertex conditions for $H$ at $v$ be given in terms of $P_{v, \rm D}, P_{v, \rm N}, P_{v, \rm R}$ and $\Lambda_v$. Assume that Hypothesis~\ref{hyp:PQ} holds for~$v$. Then the following assertions hold for each $f \in \dom h$, where the types refer to Classification~\ref{class}.
\begin{enumerate}
  \item If the conditions at $v$ are of type I~(a) then $f$ is continuous at $v$ and
		\begin{align*}
	 \big\langle \Lambda_v P_{v, \rm R} F (v), P_{v, \rm R} F (v) \big\rangle
 = 0.
\end{align*}
	\item If the conditions at $v$ are of type I~(b) then $f$ is continuous at $v$ and
	\begin{align*}
	 \big\langle \Lambda_v P_{v, \rm R} F (v), P_{v, \rm R} F (v) \big\rangle
 = \alpha_v |f (v)|^2.
\end{align*}
  \item If the conditions at $v$ are of type II~(a) then $\sum_{j = 1}^{\deg (v)} F_j (v) = 0$ and
		\begin{align*}
	 \big\langle \Lambda_v P_{v, \rm R} F (v), P_{v, \rm R} F (v) \big\rangle
 = 0.
\end{align*}
	\item If the conditions at $v$ are of type II~(b) then $f$ does not satisfy any vertex conditions at $v$ and
	\begin{align*}
	 \big\langle \Lambda_v P_{v, \rm R} F (v), P_{v, \rm R} F (v) \big\rangle
 = \frac{1}{\beta_v} \bigg| \sum_{j = 1}^{\deg (v)} F_j (v) \bigg|^2.
\end{align*}
  \item If the conditions at $v$ are of type III~(a) then $\sum_{j = 1}^{\deg (v)} F_j (v) = 0$ and
		\begin{align*}
	 \big\langle \Lambda_v P_{v, \rm R} F (v), P_{v, \rm R} F (v) \big\rangle
 = C_v \sum_{j = 1}^{\deg (v)} |F_j (v)|^2.
\end{align*}
	\item If the conditions at $v$ are of type III~(b) then $f$ does not satisfy any vertex conditions at $v$ and
	\begin{align*}
	 \big\langle \Lambda_v P_{v, \rm R} F (v), P_{v, \rm R} F (v) \big\rangle
 & = \frac{1}{D_v} \bigg( \sum_{j = 1}^{\deg (v)} |F_j (v)|^2 - \frac{1}{\deg (v)} \bigg| \sum_{j = 1}^{\deg (v)} F_j (v) \bigg|^2 \bigg).
\end{align*}
\end{enumerate}
\end{lemma}

\section{Attaching edges}
\label{sec:attaching_edges}

In this section we study the question how the spectrum of a graph Hamiltonian changes when we attach additional edges (or, more generally, whole graphs) to certain vertices of a given graph. Earlier this question was studied for the Laplacian with Kirchhoff vertex conditions and the first positive eigenvalue (the spectral gap) in~\cite{KMN13}. We provide two theorems depending on the conditions at those vertices where the additional edge or graph is attached. In order to avoid making the presentation over-complicated, in the theorems of this section we treat only the case of the Laplacian, i.e., the negative second derivative with zero potentials on the edges and discuss only adding one edge, either connecting two vertices of the original graph or a vertex of the original graph and a new vertex. We then discuss the general case of Schr\"odinger operators as well as attaching whole graphs in Remark~\ref{rem:Potentials} below.

The notion in the following definition will be used below when graph transformations lead to an increase of the degree of a vertex.

\begin{definition}
At a graph vertex $v$ of degree $d$ let vertex conditions satisfying Hypothesis~\ref{hyp:PQ} be given, that is, $P_{v, \rm D}, P_{v, \rm N}, P_{v, \rm R} \in \{0, \cP_d, \cQ_d\}$ are distinct and $\Lambda_v$ acts as multiplication with a constant in $\ran P_{v, \rm R}$. Then the {\em natural extension} of these conditions to a vertex $\widetilde v$ of degree $\widetilde d > d$ is obtained by replacing $\cP_d$ by $\cP_{\widetilde d}$ and $\cQ_d$ by $\cQ_{\widetilde d}$ and letting $\Lambda_{\widetilde v}$ be the multiplication operator in $\ran P_{\widetilde v, \rm R}$ with the constant corresponding to the same interaction strength as for $\Lambda_v$; i.e.\ multiplication with $\frac{\alpha_v}{\widetilde d}$ in case I~(b), with $\frac{\widetilde d }{\beta_v}$ in case II~(b), with $C_v$ in case III~(a) and with $\frac{1}{D_v}$ in case III~(b).
\end{definition}

We point out that proceeding from vertex conditions satisfying Hypothesis~\ref{hyp:PQ} to their natural extensions does not change the type of the conditions according to Classification~\ref{class}.

The following theorem deals with adding an edge connecting two vertices $v_1$ and $v_2$ of a graph. The admissible vertex conditions include $\delta'$-type and anti-Kirchhoff conditions.

\begin{theorem}\label{thm:monotonicity}
Let $\Gamma$ be a finite, compact metric graph, let $v_1, v_2$ be two distinct vertices of $\Gamma$ and let $H$ be the Laplacian in $L^2 (\Gamma)$ subject to arbitrary local, self-adjoint vertex conditions at each vertex $v \in V \setminus \{v_1, v_2\}$, see Proposition~\ref{prop:SAconditions}, and having at each of the vertices $v_1$ and $v_2$ conditions either of type II (a) or II (b) or III (a) according to Classification~\ref{class} (not necessarily the same type at $v_1$ and~$v_2$). Let $\widetilde \Gamma$ be the graph obtained from $\Gamma$ by adding an extra edge of an arbitrary, finite length connecting $v_1$ and $v_2$, see Figure~\ref{fig:THMmonotonicity}. Moreover, let $\widetilde H$ be the Laplacian in $L^2 (\widetilde \Gamma)$ having the same vertex conditions as $H$ on all $v \in V \setminus \{v_1, v_2\}$ and with the natural extension of the vertex conditions for $H$ at $v_1$ and $v_2$. Then
\begin{align*}
  \lambda_k (\widetilde H) \leq \lambda_k (H)
\end{align*}
holds for all $k \in \N$.
\end{theorem}

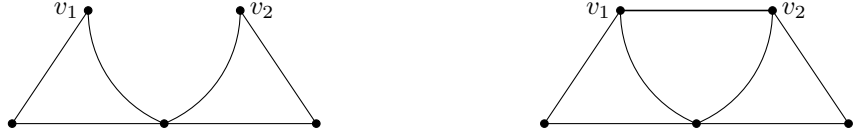
\begin{figure}[htb]
  \centering
  \begin{tikzpicture}
    \draw[fill] (0,0) circle(0.05);
    \draw[fill] (2,0) circle(0.05);
    \draw[fill] (1,1.5) circle(0.05) node[left]{$v_1$};
    \draw[fill] (4,0) circle(0.05);
    \draw[fill] (3,1.5) circle(0.05) node[right]{$v_2$};
    \draw (1,1.5)--(0,0)--(2,0);
    \draw (2,0)--(4,0)--(3,1.5);
    \draw (1,1.5) arc(180:247.38:1.625);
    \draw (3,1.5) arc(0:-67.38:1.625);
    \begin{scope}[shift={(7,0)}]
      \draw[fill] (0,0) circle(0.05);
      \draw[fill] (2,0) circle(0.05);
      \draw[fill] (1,1.5) circle(0.05) node[left]{$v_1$};
      \draw[fill] (4,0) circle(0.05);
      \draw[fill] (3,1.5) circle(0.05) node[right]{$v_2$};
      \draw (1,1.5)--(0,0)--(2,0);
      \draw (2,0)--(4,0)--(3,1.5);
      \draw (1,1.5) arc(180:247.38:1.625);
      \draw (3,1.5) arc(0:-67.38:1.625);
      \draw[semithick] (1,1.5)--(3,1.5);
    \end{scope}
  \end{tikzpicture}
  \caption{The transformation of Theorem \ref{thm:monotonicity}. Left: $\Gamma$, right: $\widetilde{\Gamma}$, i.e.\ attaching an edge connecting $v_1$ and $v_2$.}
  \label{fig:THMmonotonicity}
\end{figure}

\begin{proof}
Let us denote by $h$ and $\widetilde h$ the quadratic forms on $L^2 (\Gamma)$ and $L^2 (\widetilde \Gamma)$ corresponding to the operators $H$ and $\widetilde H$, respectively, see Proposition~\ref{prop:SAconditions}. Let $k \in \N$ and let $\cF$ be a $k$-dimensional subspace of $\dom h$ such that 
\begin{align*}
 h [f] \leq \lambda_k (H) \int_\Gamma |f|^2 \textup{d}x \quad \text{for all}~f \in \cF.
\end{align*}
For each $f \in \cF$ denote by $\widetilde f$ the extension of $f$ by zero to $\widetilde \Gamma$. Then the space $\widetilde \cF$ formed by these extensions is $k$-dimensional, and $\widetilde \cF \subset \dom \widetilde h$ as the conditions for $h$ at $v$ carry over to the corresponding conditions for $\widetilde h$. Moreover, for each $\widetilde f \in \widetilde \cF$ we have
\begin{align}\label{eq:achso}
\begin{split}
 \widetilde h [\widetilde f] & = \int_\Gamma |f'|^2 \textup{d}x + \sum_{v \in V \setminus \{v_1, v_2\}} \big\langle \Lambda_v P_{v, \rm R} F (v), P_{v, \rm R} F (v) \big\rangle \\
& \qquad + \big\langle \widetilde \Lambda_{v_1} \widetilde P_{v_1, \rm R} \widetilde F (v_1), \widetilde P_{v_1, \rm R} \widetilde F (v_1) \big\rangle + \big\langle \widetilde \Lambda_{v_2} \widetilde P_{v_2, \rm R} \widetilde F (v_2), \widetilde P_{v_2, \rm R} \widetilde F (v_2) \big\rangle.
\end{split}
\end{align}
Let us look at the term for $v_1$ in more detail. Our aim is to show
\begin{align}\label{eq:aha}
 \big\langle \widetilde \Lambda_{v_1} \widetilde P_{v_1, \rm R} \widetilde F (v_1), \widetilde P_{v_1, \rm R} \widetilde F (v_1) \big\rangle = \big\langle \Lambda_{v_1} P_{v_1, \rm R} F (v_1), P_{v_1, \rm R} F (v_1) \big\rangle.
\end{align}
Indeed, if the conditions at $v_1$ are of type II~(a) then $P_{v_1, \rm R} = 0$ and $\widetilde P_{v_1, \rm R} = 0$ so that~\eqref{eq:aha} follows. If the conditions at $v_1$ are of type II~(b) then by Lemma~\ref{lem:forms}~(iv)
\begin{align*}
 \big\langle \widetilde \Lambda_v \widetilde P_{v_1, \rm R} \widetilde F (v_1), \widetilde P_{v_1, \rm R} \widetilde F (v_1) \big\rangle & = \frac{1}{\beta_{v_1}} \bigg| \sum_{j = 1}^{d + 1} \widetilde F_j (v) \bigg|^2 = \frac{1}{\beta_{v_1}} \bigg| \sum_{j = 1}^{d} F_j (v) \bigg|^2 \\
 & = \big\langle \Lambda_{v_1} P_{v_1, \rm R} F (v_1), P_{v_1, \rm R} F (v_1) \big\rangle,
\end{align*}
where $d$ is the degree of $v_1$ in $\Gamma$. This is~\eqref{eq:aha}. Finally, if the conditions at $v_1$ are of type III~(a) then Lemma~\ref{lem:forms}~(v) gives
\begin{align*}
 \big\langle \widetilde \Lambda_v \widetilde P_{v_1, \rm R} \widetilde F (v_1), \widetilde P_{v_1, \rm R} \widetilde F (v_1) \big\rangle & = C_{v_1} \sum_{j = 1}^{d + 1} \big|\widetilde F_j (v) \big|^2 = C_{v_1} \sum_{j = 1}^{d} \big|F_j (v) \big|^2 \\
 & = \big\langle \Lambda_{v_1} P_{v_1, \rm R} F (v_1), P_{v_1, \rm R} F (v_1) \big\rangle,
\end{align*}
which is again~\eqref{eq:aha}.
Finally, combining~\eqref{eq:achso} with~\eqref{eq:aha} and its analogous counterpart for $v_1$ replaced by $v_2$ we get 
\begin{align*}
 \widetilde h [\widetilde f] = h [f] \leq \lambda_k (H) \int_\Gamma |f|^2 \textup{d}x = \lambda_k (H) \int_{\widetilde \Gamma} |\widetilde f|^2 \textup{d}x \quad \text{for all}~\widetilde f \in \widetilde \cF,
\end{align*}
which by the min-max principle implies the assertion of the theorem.
\end{proof}

The following example shows that Theorem~\ref{thm:monotonicity} may fail for other conditions satisfying Hypothesis~\ref{hyp:PQ}. For Kirchhoff conditions this example was discussed in~\cite[Example 1]{KMN13}. 

\begin{example}\label{ex:KHdelta}
Let $\Gamma$ be the graph with two vertices $v_1, v_2$ and one edge of length~$1$ connecting these two, see Figure~\ref{fig:monotonicity}. 
Let $H$ be the Laplacian in $L^2 (\Gamma)$ with a $\delta$-condition of strength $\alpha \in \R$ at $v_1$ and a Kirchhoff condition at $v_2$. 
Note that both vertices have degree one; hence, the condition at $v_1$ is a Robin boundary condition and the one at $v_2$ is Neumann. 
Note further that for $\alpha = 0$ the condition at $v_1$ is Neumann, too. 
Moreover, let $\widetilde \Gamma$ be the graph obtained from $\Gamma$ by adding another edge of length $\ell > 0$ that also connects $v_1$ to $v_2$, and let $\widetilde H$ be the Laplacian in $L^2 (\widetilde \Gamma)$ subject to the natural extensions of the conditions for $H$, namely a $\delta$-condition of strength $\alpha$ at $v_1$ and a Kirchhoff condition at $v_2$. We look at two different cases.

  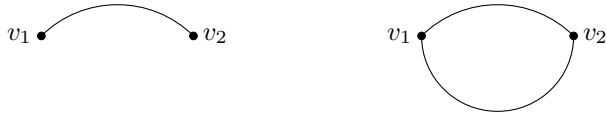
\begin{figure}[htb]
    \centering
    \begin{tikzpicture}
      \draw[fill] (0,0) circle(0.05);
      \draw[fill] (2,0) circle(0.05);
      \draw (0,0) node[left]{$v_1$};
      \draw (2,0) node[right]{$v_2$};
      \draw (2,0) arc(45:135:1.41421356237);
      \begin{scope}[shift={(5,0)}]
	\draw[fill] (0,0) circle(0.05);
	\draw[fill] (2,0) circle(0.05);
	\draw (0,0) node[left]{$v_1$};
	\draw (2,0) node[right]{$v_2$};
	\draw (2,0) arc(45:135:1.41421356237);
	\draw (2,0) arc(0:-180:1);
      \end{scope}
    \end{tikzpicture}     
    \caption{The graphs in Example~\ref{ex:KHdelta}. Left: $\Gamma$, right: $\widetilde{\Gamma}$.}
  \label{fig:monotonicity}
  \end{figure}

If $\alpha = 0$ (see~\cite[Example 1]{KMN13}) then $\lambda_k(H) = \pi^2(k-1)^2$ for all $k\in\N$, and $\widetilde{H}$ is unitarily equivalent to the Laplacian on an interval of length $1+\ell$ with periodic boundary conditions, which implies $\lambda_1(\widetilde{H}) = 0$ and $\lambda_{2k}(\widetilde{H}) = \lambda_{2k+1}(\widetilde{H}) = \frac{4\pi^2}{(1+\ell)^2} k^2$ for all $k\in\N$. 
Hence for $\ell<1$ we observe $\lambda_2(\widetilde{H}) > \lambda_2(H)$, whereas for $\ell\geq1$ we obtain $\lambda_k(\widetilde{H}) \leq \lambda_k(H)$ for all $k\in\N$ (and even a strict inequality for $k>2$; for $\ell>1$ the inequality is strict also for $k=2$).

If $\alpha = 1$ then simple calculations yield $\lambda_1(H) \approx 0.74017$. Furthermore, letting $\ell = 0.1$ one obtains $\lambda_1 (\widetilde{H}) \approx 0.83156$, that is, $\lambda_1 (H) < \lambda_1 (\widetilde H)$. Hence, adding an edge may increase the eigenvalues also for $\alpha \neq 0$.
\end{example}

We would like to point out that we did not find an example for conditions of type III~(b) violating the assertion of Theorem~\ref{thm:monotonicity}. We provide a further example that shows that the inequality in Theorem \ref{thm:monotonicity} can be strict for all $k \in \N$ simultaneously.

\begin{example}
Let $\Gamma$ and $\widetilde \Gamma$ be as in the previous example, where the edge length in $\Gamma$ is $1$ and the edge lengths in $\widetilde \Gamma$ are $1$ and $\ell > 0$, and let $H$ and $\widetilde H$ be the Laplacians in $L^2 (\Gamma)$ and $L^2 (\widetilde \Gamma)$, respectively, with type II~(a) (anti-Kirchhoff) coupling conditions at both $v_1$ and $v_2$. For $\Gamma$ this results in Dirichlet boundary conditions and, hence, $\lambda_k(H) = \pi^2k^2$ for all $k\in\N$. 
Moreover, $\widetilde{H}$ is unitarily equivalent to the Laplacian $\widehat{H}$ on the interval $[0, 1+\ell]$ with conditions $f (1-) + f (1+) = 0$, $-f'(1-) = f'(1+)$, $f(0)+ f(1+\ell) = 0$ and $f'(0) = - f'(1+\ell)$. Hence, $f$ is an eigenfunction of $\widehat{H}$ if and only if $g$ defined by $g = f$ on $[0,\ell)$ and $g = -f$ on $(\ell,1+\ell]$ is an eigenfunction of the Laplacian on $[0,1+\ell]$ with periodic boundary conditions.
  We conclude that $\lambda_1(\widetilde{H}) = 0$ and
  $\lambda_{2k}(\widetilde{H}) = \lambda_{2k+1}(\widetilde{H}) = \frac{4\pi^2}{(1+\ell)^2} k^2$ for all $k\in\N$.
  Thus, $\lambda_k(\widetilde{H}) < \lambda_k(H)$ for all $k\in\N$, independent of the choice of $\ell$.
\end{example}

In the next theorem we show that adding an edge connecting an old vertex to a new vertex of degree one does in most cases lead to a non-increase of all eigenvalues. For the special case of Kirchhoff vertex conditions this was discussed in~\cite[Theorem~2]{KMN13}, see also~\cite[Proposition~3.1]{R17}. For the sake of completeness we indicate the proof also for this case.

\begin{theorem}\label{thm:monotonicity2}
Let $\Gamma$ be a finite, compact metric graph, let $v_1$ be a vertex of $\Gamma$ and let $H$ be the Laplacian in $L^2 (\Gamma)$ subject to arbitrary local, self-adjoint vertex conditions at each vertex $v \in V \setminus \{v_1\}$, see Proposition~\ref{prop:SAconditions}, and with a condition satisfying Hypothesis~\ref{hyp:PQ} at $v_1$. 
Let $\widetilde \Gamma$ be the graph obtained from $\Gamma$ by adding an extra edge of an arbitrary, finite length connecting $v_1$ with a new vertex $v_2$ (i.e., $v_2 \notin V$ and $v_2$ has degree one in $\widetilde \Gamma$), see Figure \ref{fig:THMmonotonicity2}. 
Moreover, let $\widetilde H$ be the Laplacian in $L^2 (\widetilde \Gamma)$ having the same vertex conditions as $H$ on all $v \in V \setminus \{v_1\}$, with the natural extension of the vertex conditions for $H$ at $v_1$ and with any self-adjoint, local conditions at $v_2$. 
If the conditions at $v_1$ are of type I~(a), I~(b) or III~(b) we assume in addition that the condition of $\widetilde H$ at $v_2$ is a $\delta$ (i.e.\ Robin) condition with $\alpha_{v_2} \leq 0$. Then
\begin{align*}
  \lambda_k (\widetilde H) \leq \lambda_k (H) 
\end{align*}
holds for all $k \in \N$. 
\end{theorem}

\begin{figure}[htb]
  \centering
  \begin{tikzpicture}
    \draw[fill] (0,0) circle(0.05);
    \draw[fill] (2,0) circle(0.05);
    \draw[fill] (1,1.5) circle(0.05) node[left]{$v_1$};
    \draw[fill] (4,0) circle(0.05);
    \draw[fill] (3,1.5) circle(0.05);
    \draw (1,1.5)--(0,0)--(2,0);
    \draw (2,0)--(4,0)--(3,1.5);
    \draw (1,1.5) arc(180:247.38:1.625);
    \draw (3,1.5) arc(0:-67.38:1.625);
    \begin{scope}[shift={(7,0)}]
      \draw[fill] (0,0) circle(0.05);
      \draw[fill] (2,0) circle(0.05);
      \draw[fill] (1,1.5) circle(0.05) node[left]{$v_1$};
      \draw[fill] (4,0) circle(0.05);
      \draw[fill] (3,1.5) circle(0.05);
      \draw (1,1.5)--(0,0)--(2,0);
      \draw (2,0)--(4,0)--(3,1.5);
      \draw (1,1.5) arc(180:247.38:1.625);
      \draw (3,1.5) arc(0:-67.38:1.625);
      \draw[fill] (2,2) circle(0.05) node[right]{$v_2$};
      \draw[semithick] (1,1.5)--(2,2);
    \end{scope}
  \end{tikzpicture}
  \caption{The transformation of Theorem \ref{thm:monotonicity2}. Left: $\Gamma$, right: $\widetilde{\Gamma}$, i.e.\ adding an extra edge connecting $v_1$ and a new vertex $v_2$.}
  \label{fig:THMmonotonicity2}
\end{figure}
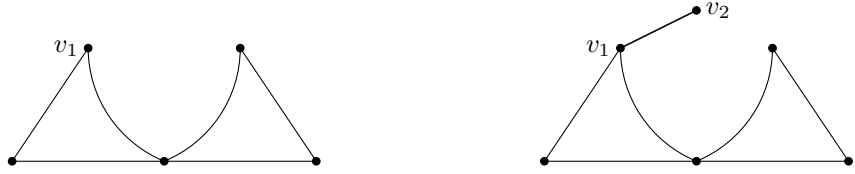

\begin{proof}
If the condition at $v_1$ is of type II~(a), II~(b) or III~(a) then the proof is literally the same as for Theorem~\ref{thm:monotonicity}, and the result is independent of the boundary condition at the new vertex $v_2$. For the remaining cases let $k \in \N$ and let $\cF$ be as in the proof of Theorem~\ref{thm:monotonicity}. If the condition at $v_1$ is of type~I (Kirchhoff or $\delta$) then each $f \in \cF$ is continuous at $v_1$ and we define $\widetilde f$ to be the extension of $f$ to $\widetilde \Gamma$ having the constant value $f (v_1)$ on the new edge, and $\widetilde \cF$ to be the collection of all such extensions of $f \in \cF$. As the condition of $\widetilde H$ at $v_2$ is $\delta$, it follows $\widetilde \cF \subset \dom \widetilde h$. Moreover, by Lemma~\ref{lem:forms}~(i) or~(ii),
\begin{align}\label{eq:nullmal}
\begin{split}
 \big\langle \widetilde \Lambda_{v_1} \widetilde P_{v_1, \rm R} \widetilde F (v_1), \widetilde P_{v_1, \rm R} \widetilde F (v_1) \big\rangle & = \alpha_{v_1} |\widetilde f (v_1)|^2 = \alpha_{v_1} |f (v_1)|^2 \\
& = \big\langle \Lambda_{v_1} P_{v_1, \rm R} F (v_1), P_{v_1, \rm R} F (v_1) \big\rangle
\end{split}
\end{align}
holds for the coefficient $\alpha_{v_1} \in \R$ of the condition at $v_1$. Moreover, at $v_2$ we have
\begin{align}\label{eq:einmal}
 \big\langle \widetilde \Lambda_{v_2} \widetilde P_{v_2, \rm R} \widetilde F (v_2), \widetilde P_{v_2, \rm R} \widetilde F (v_2) \big\rangle = \alpha_{v_2} |\widetilde f (v_2)|^2 \leq 0.
\end{align}
We plug~\eqref{eq:nullmal} and~\eqref{eq:einmal} into~\eqref{eq:achso}, which remains valid in this situation, and obtain 
\begin{align}\label{eq:naBitte}
 \widetilde h [\widetilde f] \leq h [f] \leq \lambda_k (H) \int_\Gamma |f|^2 \textup{d}x \leq \lambda_k (H) \int_{\widetilde \Gamma} |\widetilde f|^2 \textup{d}x \quad \text{for all}~\widetilde f \in \widetilde \cF.
\end{align}
If the condition at $v_2$ is of type III~(b) then for each $f \in \cF$ we denote by $\widetilde f$ the extension to $\widetilde \Gamma$ being constantly equal to
\begin{align*}
 \frac{1}{d} \sum_{j = 1}^{d} F_j (v_1)
\end{align*}
on the new edge, where $d$ is the degree of $v_1$ in $\Gamma$. As above, we denote by $\widetilde \cF$ the space of these extensions of all functions in $\cF$. Then by Lemma~\ref{lem:forms}~(vi)
\begin{align*}
 & \big\langle \widetilde \Lambda_{v_1} \widetilde P_{v_1, \rm R} \widetilde F (v_1), \widetilde P_{v_1, \rm R} \widetilde F (v_1) \big\rangle 
 = \frac{1}{D_{v_1}} \bigg( \sum_{j = 1}^{d + 1} |\widetilde F_j (v_1)|^2 - \frac{1}{d + 1} \bigg| \sum_{j = 1}^{d + 1} \widetilde F_j (v_1) \bigg|^2 \bigg) \\
 & = \frac{1}{D_{v_1}} \bigg( \sum_{j = 1}^{d} |F_j (v_1)|^2 + \frac{1}{d^2} \bigg| \sum_{j = 1}^{d} F_j (v_1) \bigg|^2 - \frac{1}{d + 1} \bigg| \sum_{j = 1}^{d} F_j (v_1) + \frac{1}{d} \sum_{j = 1}^{d} F_j (v_1) \bigg|^2 \bigg) \\
 & = \frac{1}{D_{v_1}} \bigg( \sum_{j = 1}^d |F_j (v_1)|^2 - \frac{1}{d} \bigg|\sum_{j = 1}^{d} F_j (v_1)\bigg|^2 \bigg) \\
 & = \big\langle \Lambda_{v_1} P_{v_1, \rm R} F (v_1), P_{v_1, \rm R} F (v_1) \big\rangle.
\end{align*}
Together with
\begin{align*}
 \big\langle \widetilde \Lambda_{v_2} \widetilde P_{v_2, \rm R} \widetilde F (v_2), \widetilde P_{v_2, \rm R} \widetilde F (v_2) \big\rangle = \alpha_{v_2} |\widetilde f (v_1)|^2 \leq 0
\end{align*}
and~\eqref{eq:achso} we arrive again at~\eqref{eq:naBitte}. As $\dim \widetilde \cF = k$, the min-max principle implies the assertion of the theorem.
\end{proof}

The following example shows that in case we impose a Robin boundary condition with a positive coefficient or a Dirichlet boundary condition at the new vertex $v_2$, the assertion of the previous theorem may be false. 

\begin{example}
\label{ex:attach_new_vertex}
Let $\Gamma$ be the graph consisting of two vertices $v_0$ and $v_1$ and one edge connecting the two vertices, which is parametrised by the interval $(0,1)$; cf.~Figure~\ref{fig:attach_new_vertex}.
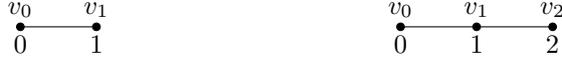
\begin{figure}[htb]
  \centering
  \begin{tikzpicture}
    \draw[fill] (0,0) circle(0.05) node[above]{$v_0$};
    \draw (0,0) node[below]{$0$};
    \draw[fill] (1,0) circle(0.05) node[above]{$v_1$};
    \draw (1,0) node[below]{$1$};  
    \draw (0,0)--(1,0);
    \begin{scope}[shift={(5,0)}]
     \draw[fill] (0,0) circle(0.05) node[above]{$v_0$};
      \draw (0,0) node[below]{$0$};
      \draw[fill] (1,0) circle(0.05) node[above]{$v_1$};
      \draw (1,0) node[below]{$1$}; 
      \draw (0,0)--(1,0);
      \draw[fill] (2,0) circle(0.05) node[above]{$v_2$};
      \draw (2,0) node[below]{$2$};
      \draw (1,0)--(2,0);
    \end{scope}
  \end{tikzpicture}
  \caption{The graphs in Example \ref{ex:attach_new_vertex}. Left: $\Gamma$, right: $\widetilde{\Gamma}$.}
  \label{fig:attach_new_vertex}
\end{figure}
Let us impose Kirchhoff conditions at $v_0$ and $v_1$. Since the degree is one in both cases, these conditions correspond to Neumann boundary conditions for $H$ at $0$ and at $1$. Hence, $\lambda_k (H) = (k-1)^2\pi^2$ for all $k\in\N$. Now, let us add an edge of length one, which connects $v_1$ with $v_2$, see Figure \ref{fig:attach_new_vertex}.
At the vertex $v_2$ we impose either the condition $f' (2) + \alpha f (2) = 0$ for some $\alpha > 0$ or a Dirichlet condition. Then the spectrum of $\widetilde H$ is nonnegative and it is easy to see by solving the respective boundary value problem that $\lambda_1 (\widetilde H) \neq 0$. Hence
\begin{align*}
 \lambda_1 (\widetilde H) > 0 = \lambda_1 (H),
\end{align*}
which shows that the assertion of Theorem~\ref{thm:monotonicity2} is not valid here.
\end{example}

Let us add some concluding remarks on generalizations of the results of this section.

\begin{remark}\label{rem:Potentials}
\begin{enumerate}
 \item[(a)] As the proofs show, the statements of Theorem~\ref{thm:monotonicity} and Theorem~\ref{thm:monotonicity2} remain true if the Laplacian on $\Gamma$ is replaced by a Schr\"odinger operator with a real-valued potential $q\in L^\infty(\Gamma)$. Furthermore, the statement Theorem~\ref{thm:monotonicity} remains valid if an arbitrary bounded, measurable, real-valued potential is introduced on the new edge, and the same is true for Theorem~\ref{thm:monotonicity2} if the conditions at $v_1$ are of type II~(a), II~(b) or III~(a). In the remaining cases of Theorem~\ref{thm:monotonicity2} it is easy to see that the statement remains true if the potential $q_{\widetilde e}$ on the new edge $\widetilde e$ satisfies 
 \begin{align*}
  \int_0^{L (\widetilde e)} q_{\widetilde e} \, \textup{d}x \leq 0.
 \end{align*}
 \item[(b)] Instead of just adding one edge in Theorem \ref{thm:monotonicity}, the result remains valid when one attaches a whole compact metric graph to a subset $V_0 \subset V$ provided the conditions at each vertex of $V_0$ have one of the types II~(a), II~(b) or III~(a). Also in the situation of Theorem \ref{thm:monotonicity2} we can add a whole compact metric graph to one vertex, but in that case appropriate assumptions on the vertex conditions of the attached graph need to be imposed. We do not discuss this here in more detail.
\end{enumerate}
\end{remark}

\section{Joining vertices}
\label{sec:joining_vertices}

In this section we study the behavior of eigenvalues when joining two vertices (with the same type of conditions) and merging their coupling conditions appropriately. In the following we say that two vertices $v_1$, $v_2$ of a graph are {\em joined} to one vertex $v_0$ if $v_1$ and $v_2$ are removed from $\Gamma$ and, instead, a vertex $v_0$ with
\begin{align*}
 E_{v_0, \rm i}:=E_{v_1, \rm i}\cup E_{v_2, \rm i} \quad \text{and} \quad E_{v_0, \rm t} := E_{v_1, \rm t}\cup E_{v_2, \rm t}
\end{align*}
is introduced. We point out that the process of joining two vertices does not affect the edge set of a metric graph. In particular, each function on the original graph can be identified naturally with a function on the new graph where $v_1$ and $v_2$ are joined to $v_0$, and vice versa. In particular, we will identify the spaces $L^2 (\Gamma)$ and $L^2 (\widetilde \Gamma)$ if $\widetilde \Gamma$ was obtained from $\Gamma$ by joining two vertices.

\begin{figure}[htb]
  \centering
  \begin{tikzpicture}
    \draw[fill] (0,0) circle(0.05);
    \draw[fill] (2,0) circle(0.05);
    \draw[fill] (1,1.5) circle(0.05) node[left]{$v_1$};
    \draw[fill] (4,0) circle(0.05);
    \draw[fill] (3,1.5) circle(0.05);
    \draw[fill] (3,1.5) circle(0.05) node[right]{$v_2$};
    \draw (1,1.5)--(0,0)--(2,0);
    \draw (2,0)--(4,0)--(3,1.5);
    \draw (1,1.5) arc(180:247.38:1.625);
    \draw (3,1.5) arc(0:-67.38:1.625);
    \begin{scope}[shift={(7,0)}]
      \draw[fill] (0,0) circle(0.05);
      \draw[fill] (2,0) circle(0.05);
      \draw[fill] (1,1.5) circle(0.05) node[above]{$v_0$};
      \draw (1,1.5)--(0,0)--(2,0);
      \draw (1,1.5) arc(180:247.38:1.625);
      \begin{scope}[rotate around={67.38:(2,0)}]
	\draw (3,1.5) arc(0:-67.38:1.625);
	\draw[fill] (4,0) circle(0.05);
	\draw[fill] (3,1.5) circle(0.05);
	\draw (2,0)--(4,0)--(3,1.5);
      \end{scope}
    \end{scope}
  \end{tikzpicture}
  \caption{The transformation of Theorems \ref{thm:joining}, \ref{thm:joining_delta'} and~\ref{thm:joining_ExnerConditions}}. Left: $\Gamma$, right: $\widetilde{\Gamma}$, i.e.\ joining the vertices $v_1$ and $v_2$ to a new vertex $v_0$.
  \label{fig:THMjoining}
\end{figure}
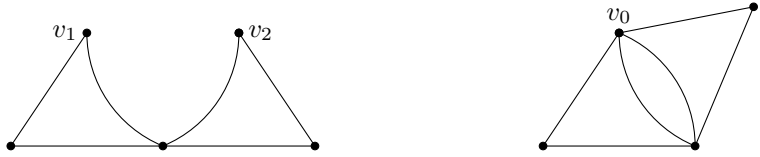

We start with the case of conditions of type I, i.e., Kirchhoff or $\delta$-conditions. This situation was treated in~\cite[Theorem~2]{KKT16}. For completeness we include its simple proof here. We allow nonzero potentials on the edges as well as arbitrary self-adjoint conditions at the vertices that are not changed, which does not complicate the proof.

\begin{theorem}
\label{thm:joining}
Let $\Gamma$ be a finite, compact metric graph and let $H$ be a Schr\"odinger operator in $L^2 (\Gamma)$ with real-valued potential $q\in L^\infty(\Gamma)$ and local, self-adjoint vertex conditions at the vertices; cf.~Proposition~\ref{prop:SAconditions}. Assume that $v_1, v_2$ are two distinct vertices of $\Gamma$ and that the vertex conditions of $H$ at $v_1$ and $v_2$ are of type~I according to Classification~\ref{class}, i.e.\ of $\delta$-type with coefficients $\alpha_{v_1}, \alpha_{v_2} \in \R$ (coefficient zero corresponds to a Kirchhoff condition). Denote by $\widetilde \Gamma$ the graph obtained from $\Gamma$ by joining $v_1$ and $v_2$ to form one single vertex $v_0$. Let $\widetilde H$ be the self-adjoint Schr\"odinger operator in $L^2 (\widetilde \Gamma)$ with potential $q$ having the same vertex conditions as $H$ at all vertices apart from $v_0$ and satisfying a $\delta$-type condition with coefficient $\alpha_{v_0} :=\alpha_{v_1} + \alpha_{v_2}$ at $v_0$. Then
\begin{align}\label{eq:yeah}
 \lambda_k (H) \leq \lambda_k (\widetilde H)
\end{align}
holds for all $k\in\N$.
\end{theorem}

\begin{proof}
Let $h$ and $\widetilde h$ be the quadratic forms corresponding to the operators $H$ and $\widetilde H$, respectively (see Proposition \ref{prop:SAconditions}). For the inequality~\eqref{eq:yeah} it suffices to show inequality of the quadratic forms, i.e., $\dom \widetilde h \subset \dom h$ and $h [f] \leq \widetilde h [f]$ for all $f \in \dom \widetilde h$. In fact, each function in $\dom \widetilde h$ is continuous at each vertex of $\widetilde \Gamma$ which clearly implies continuity at each vertex of $\Gamma$. In particular, each $f \in \dom \widetilde h$ satisfies $f (v_0) = f (v_1) = f (v_2)$, and by Lemma~\ref{lem:forms}~(i) or~(ii) we get
\begin{align*}
 \widetilde h [f] - h [f] & = \alpha_{v_0} |f (v_0)|^2 - \alpha_{v_1} |f (v_1)|^2 - \alpha_{v_2} |f (v_2)|^2 = 0,
\end{align*}
which leads to the assertion.
\end{proof}

In the next theorem we show that joining two vertices of type II can have a different effect on the eigenvalues, depending on the signs of the coupling coefficients. Recall that $f$ satisfies conditions of type~II at some vertex $v$ if $F' (v)$ is a constant vector (let us call an arbitrary component $f' (v)$)  and
\begin{align}\label{eq:typeII}
 \sum_{j = 1}^{\deg (v)} F_j (v) = \beta_v f' (v).
\end{align}
These are $\delta'$-type conditions for $\beta_v \neq 0$ and anti-Kirchhoff conditions for $\beta_v = 0$. 

\begin{theorem}\label{thm:joining_delta'}
Let $\Gamma$ be a finite, compact metric graph and let $H$ be a Schr\"odinger operator in $L^2 (\Gamma)$ with real-valued potential $q\in L^\infty(\Gamma)$ and local, self-adjoint vertex conditions at the vertices; cf.~Proposition~\ref{prop:SAconditions}. Assume that $v_1, v_2$ are two distinct vertices of $\Gamma$ and that each of the vertex conditions of $H$ at $v_1$ and $v_2$ is of type II according to Classification~\ref{class}, with coefficients $\beta_{v_1}, \beta_{v_2} \in \R$. Denote by $\widetilde \Gamma$ the graph obtained from $\Gamma$ by joining $v_1$ and $v_2$ to form one single vertex $v_0$ (see Figure~\ref{fig:THMjoining}) and let $\widetilde H$ be the self-adjoint Schr\"odinger operator in $L^2 (\widetilde \Gamma)$ with potential $q$, having the same vertex conditions as $H$ at all vertices apart from $v_0$ and satisfying conditions of the form~\eqref{eq:typeII} at $v = v_0$ with coefficient $\beta_{v_0} := \beta_{v_1} + \beta_{v_2}$ at $v_0$. Then the following assertions hold.
\begin{enumerate}
 \item If $\beta_{v_1},\beta_{v_2} > 0$ then $\lambda_k (\widetilde H) \leq \lambda_k (H)$ for all $k \in \N$.
 \item If $\beta_{v_1},\beta_{v_2} < 0$ then $\lambda_k (H) \leq \lambda_k (\widetilde H)$ for all $k \in \N$.
 \item If $\beta_{v_1}\cdot \beta_{v_2} < 0$ and $\beta_{v_0} > 0$ then $\lambda_k (H) \leq \lambda_k (\widetilde H)$ for all $k \in \N$.
 \item If $\beta_{v_1}\cdot \beta_{v_2} < 0$ and $\beta_{v_0} < 0$ then $\lambda_k (\widetilde H) \leq \lambda_k (H)$ for all $k \in \N$.
 \item If $\beta_{v_1}\cdot \beta_{v_2} < 0$ and $\beta_{v_0} = 0$ then $\lambda_k (H) \leq \lambda_k (\widetilde H)$ for all $k \in \N$.
 \item If $\beta_{v_1} \cdot \beta_{v_2} = 0$ then $\lambda_k (\widetilde H) \leq \lambda_k (H)$ for all $k \in \N$.
\end{enumerate}
\end{theorem}

The proof of this theorem relies on the following simple lemma.

\begin{lemma}
\label{lem:nabitte}
  Let $a,b\in\C$, $p,q,r>0$, $\frac{1}{p} + \frac{1}{q} = \frac{1}{r}$. Then
  \[r |a + b|^2 \leq p |a|^2 + q |b|^2.\]
\end{lemma}

\begin{proof}
Let $y>0$. Then the Cauchy--Schwarz inequality yields
\begin{align*}
\begin{split}
 |a + b|^2 & = \left| \left\langle \binom{\sqrt{y} a}{\frac{b}{\sqrt{y}}}, \binom{\frac{1}{\sqrt{y}}}{\sqrt{y}} \right\rangle \right|^2 \leq  \left(y |a|^2 + \frac{|b|^2}{y} \right) \left( \frac{1}{y} + y \right) \\
 & = \left( 1 + y^2 \right) |a|^2 + \left( \frac{1}{y^2} + 1 \right) |b|^2.
\end{split}
\end{align*}
Let $y:=\sqrt{\frac{p}{q}}>0$. Then $1+y^2 = 1+\frac{p}{q} = \frac{p}{r}$ and $\frac{1}{y^2}+1 = \frac{q}{p}+1 = \frac{q}{r}$, and the desired inequality follows.
\end{proof}

\begin{proof}[Proof of Theorem \ref{thm:joining_delta'}]
As in the proof of the previous theorem we show inequalities for the quadratic forms $h$ and $\widetilde h$ corresponding to $H$ and $\widetilde H$, respectively, under the different conditions. These form inequalities will immediately imply the statements. Let us first discuss the assertions (i)--(iv). These are the cases where the conditions for $H$ at both $v_1$ and $v_2$ are of type II~(b) and the same holds for the conditions for $\widetilde H$ at $v_0$. In particular, for functions in the form domain of $h$ no conditions are imposed at $v_1$ and $v_2$, and the same holds for $\widetilde h$ and $v_0$. Hence the domains of $h$ and $\widetilde h$ coincide. For $f \in \dom h = \dom \widetilde h$ let us set
\begin{align}\label{eq:ab}
 a:=\sum_{j = 1}^{\deg (v_1)} F_j (v_1) \quad \text{and} \quad b:=\sum_{j = 1}^{\deg (v_2)} F_j (v_2).
\end{align}
Then 
\begin{align*}
 a+b = \sum_{j = 1}^{\deg (v_0)} F_j (v_0)
\end{align*}
and
\begin{align}\label{eq:formen}
 \widetilde h [f] - h [f] & = \frac{1}{\beta_{v_0}} |a + b|^2 - \frac{1}{\beta_{v_1}} |a|^2 - \frac{1}{\beta_{v_2}} |b|^2.
\end{align}
Now the assertions (i)--(iv) can be derived as follows: 

(i) If both $\beta_{v_1}, \beta_{v_2} > 0$ (and hence $\beta_{v_0} > 0$) we set $p := \frac{1}{\beta_{v_1}}$, $q := \frac{1}{\beta_{v_2}}$ and $r := \frac{1}{\beta_{v_0}}$, and Lemma~\ref{lem:nabitte} together with~\eqref{eq:formen} yields $\widetilde h [f] \leq h [f]$. Hence, $\widetilde h \leq h$.

(ii)  If $\beta_{v_1}, \beta_{v_2} < 0$ (and hence $\beta_{v_0} < 0$) then Lemma \ref{lem:nabitte} applied with $p = - \frac{1}{\beta_{v_1}}$, $q = - \frac{1}{\beta_{v_2}}$ and $r = - \frac{1}{\beta_{v_0}}$ and~\eqref{eq:formen} yield $h [f] \leq \widetilde h [f]$ and, hence, $h \leq \widetilde h$.

(iii) If $\beta_{v_1} > 0$ and $\beta_{v_2} < 0$ such that $\beta_{v_0} = \beta_{v_1} + \beta_{v_2} > 0$ we set $p:=\frac{1}{\beta_{v_0}}$, $q:=-\frac{1}{\beta_{v_2}}$ and $r:=\frac{1}{\beta_{v_1}}$, and by Lemma \ref{lem:nabitte} we obtain
\begin{align*}
 \frac{1}{\beta_{v_1}} |a|^2 = r|a|^2 = r|a + b - b|^2\leq p|a+b|^2 + q|b|^2 = \frac{1}{\beta_{v_0}} |a + b|^2 - \frac{1}{\beta_{v_2}} |b|^2.
\end{align*}
Hence~\eqref{eq:formen} yields $\widetilde h [f] \geq h [f]$, that is, $h \leq \widetilde h$.

(iv) If $\beta_{v_1}>0$ and $\beta_{v_2}<0$ such that $\beta_{v_0} = \beta_{v_1}+\beta_{v_2}<0$ we set $p:=-\frac{1}{\beta_{v_0}}$, $q:=\frac{1}{\beta_{v_1}}$ and $r:=-\frac{1}{\beta_{v_2}}$. Then by Lemma~\ref{lem:nabitte} we obtain
\begin{align*}
 - \frac{1}{\beta_{v_2}} |b|^2 = r|b|^2 = r|a + b - a|^2\leq p|a+b|^2 + q|a|^2 = - \frac{1}{\beta_{v_0}} |a + b|^2 + \frac{1}{\beta_{v_1}} |a|^2.
\end{align*}
Together with~\eqref{eq:formen} this gives $\widetilde h [f] \leq h [f]$, that is, $\widetilde h \leq h$.

It remains to treat the cases where zero appears as a coefficient. Under the conditions of (v), let $f \in \dom \widetilde h$. Then clearly $f \in \dom h$ (which does not require any conditions at $v_1$ or $v_2$) and with $a$ and $b$ defined in~\eqref{eq:ab} we have $a + b = 0$. Hence
\begin{align*}
 \widetilde h [f] - h [f] = - \frac{1}{\beta_{v_1}} |a|^2 - \frac{1}{\beta_{v_2}} |b|^2 = - \Big( \frac{1}{\beta_{v_1}} + \frac{1}{\beta_{v_2}} \Big) |a|^2 = 0
\end{align*}
as $\beta_{v_2} = - \beta_{v_1}$. Hence $h \leq \widetilde h$, which implies (v).

For the remaining assertion (vi) let us first look at the case $\beta_{v_1} = 0$, $\beta_{v_2} \neq 0$. In this situation let $f \in \dom h$. Then clearly $f \in \dom \widetilde h$ (no condition at $v_0$ is required since $\beta_{v_0} \neq 0$) and in the above notation we have $a = 0$. Then
\begin{align*}
 \widetilde h [f] - h [f] = \frac{1}{\beta_{v_0}} |a + b|^2 - \frac{1}{\beta_{v_2}} |b|^2 = \frac{1}{\beta_{v_2}} |b|^2 - \frac{1}{\beta_{v_2}} |b|^2 = 0.
\end{align*}
Thus $\widetilde h \leq h$. The case $\beta_{v_1} \neq 0, \beta_{v_2} = 0$ is analogous. In the final case that $\beta_{v_0} = \beta_{v_1} = \beta_{v_2} = 0$ let $f \in \dom h$. Then in the above notation we have $a = b = 0$ and, in particular, $a + b = 0$ which implies $f \in \dom \widetilde h$. Moreover, we see directly $h [f] - \widetilde h [f] = 0$. Thus $\widetilde h \leq h$. This completes the proof.
\end{proof}

We provide an example where strict inequality appears in the setting of Theorem~\ref{thm:joining_delta'}~(vi).

\begin{example}
  Let $\Gamma$ be the graph with two vertices and one edge of length $1$ connecting these two and let $H$ be the Laplacian in $L^2(\Gamma)$ with type II~(a) (anti-Kirchhoff) couplings at both vertices; since the degree of the vertices is $1$, this results in Dirichlet boundary conditions. Hence, $\lambda_k(H) = \pi^2k^2$ for $k\in\N$. Now, let us join the two vertices, see Figure \ref{fig:joining_example}.
  
  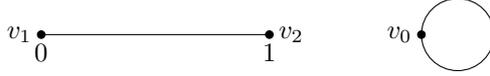
\begin{figure}[htb]
    \centering
    \begin{tikzpicture}
      \draw[fill] (0,0) circle(0.05) node[left]{$v_1$};
      \draw (0,0) node[below]{$0$};
      \draw[fill] (3,0) circle(0.05) node[right]{$v_2$};
      \draw (3,0) node[below]{$1$};
      \draw (0,0)--(3,0);
      \begin{scope}[shift={(5,0)}]
	\draw[fill] (0,0) circle(0.05) node[left]{$v_0$};
	\draw (0,0) arc(180:540:{3/(2*pi)});
      \end{scope}
    \end{tikzpicture}
    \caption{Left: $\Gamma$, right: $\widetilde{\Gamma}$.}
    \label{fig:joining_example}
  \end{figure}
  
  Then $\widetilde{H}$ is unitarily equivalent to the Laplacian on the interval $[0,1]$ with conditions $f'(0) = -f'(1)$ and $f(0) = -f(1)$, i.e.\ anti-periodic boundary conditions.
  Hence, $\lambda_{2k-1}(\widetilde{H}) = \lambda_{2k}(\widetilde{H}) = \pi^2(2k-1)^2$ for all $k\in\N$. Thus, the inequality between the eigenvalues is strict for the even indices and an equality for the odd indices.
\end{example}

The following last theorem of this section deals with joining vertices with conditions of type III~(a) or (b). In view of the form of these conditions summing up the corresponding coefficients does not seem appropriate. Instead we join vertices with the same strength of interaction.

\begin{theorem}\label{thm:joining_ExnerConditions}
Let $\Gamma$ be a finite, compact metric graph and let $H$ be a Schr\"odinger operator in $L^2 (\Gamma)$ with real-valued potential $q\in L^\infty(\Gamma)$ and local, self-adjoint coupling conditions at the vertices; cf.~Proposition~\ref{prop:SAconditions}. Assume that $v_1, v_2$ are two distinct vertices of $\Gamma$ such that the vertex conditions of $H$ at $v_1$ and $v_2$ are either both of type III~(a) according to Classification~\ref{class}, with coefficients $C_{v_1} = C_{v_2} \in \R$ or both of type III~(b) with coefficients $D_{v_1} = D_{v_2} \neq 0$. Denote by $\widetilde \Gamma$ the graph obtained from $\Gamma$ by joining $v_1$ and $v_2$ to form one single vertex $v_0$. Let $\widetilde H$ be the self-adjoint Schr\"odinger operator in $L^2 (\widetilde \Gamma)$ with potential $q$ having the same vertex conditions as $H$ at all vertices apart from $v_0$ and satisfying, at $v_0$, conditions of the same form as $H$ satisfies at $v_1$ and $v_2$, with $C_{v_0} := C_{v_1} = C_{v_2}$ or $D_{v_0} := D_{v_1} = D_{v_2}$, respectively, at $v_0$. Then the following assertions hold.
\begin{enumerate}
 \item If the conditions at $v_1, v_2, v_0$ are of type III~(a) then $\lambda_k (\widetilde H) \leq \lambda_k (H)$ for all $k \in \N$.
 \item If the conditions at $v_1, v_2, v_0$ are of type III~(b) with coefficient $D_{v_0} > 0$ then $\lambda_k (H) \leq \lambda_k (\widetilde H)$ for all $k \in \N$.
 \item If the conditions at $v_1, v_2, v_0$ are of type III~(b) with coefficient $D_{v_0} < 0$ then $\lambda_k (\widetilde H) \leq \lambda_k (H)$ for all $k \in \N$.
\end{enumerate}
\end{theorem}

\begin{proof}
Let again $h$ and $\widetilde h$ be the quadratic forms corresponding to $H$ and $\widetilde H$, respectively. Let first the conditions at $v_1, v_2, v_0$ be of type III~(a), let $f \in \dom h$, and let us set again
\begin{align}\label{eq:abNochmal}
 a:=\sum_{j = 1}^{\deg (v_1)} F_j (v_1) \quad \text{and} \quad b:=\sum_{j = 1}^{\deg (v_2)} F_j (v_2).
\end{align}
Then $a = b = 0$ and, hence, $a + b = 0$, so that $f \in \dom \widetilde h$. Moreover, by Lemma~\ref{lem:forms}~(v),
\begin{align*}
 \widetilde h [f] - h [f] & = C_{v_0} \sum_{j = 1}^{\deg (v_0)} |F_j (v_0)|^2 - C_{v_1} \sum_{j = 1}^{\deg (v_1)} |F_j (v_1)|^2 - C_{v_2} \sum_{j = 1}^{\deg (v_2)} |F_j (v_2)|^2 \\
 & = 0
\end{align*}
as $C_{v_0} = C_{v_1} = C_{v_2}$. Hence $\widetilde h \leq h$, which implies~(i).

Assume now that the conditions at $v_1, v_2, v_0$ are of type III~(b). Then the domains of $h$ and $\widetilde h$ coincide, and for each $f \in \dom h = \dom \widetilde h$ and $a, b$ defined in~\eqref{eq:abNochmal}, Lemma~\ref{lem:forms}~(vi) yields
\begin{align*}
 \widetilde h [f] - h [f] & = \frac{1}{D_{v_0}} \bigg( \sum_{j = 1}^{\deg (v_0)} |F_j (v_0)|^2 - \frac{1}{\deg (v_0)} |a + b|^2 \bigg) \\
 & \qquad - \frac{1}{D_{v_1}} \bigg( \sum_{j = 1}^{\deg (v_1)} |F_j (v_1)|^2 - \frac{1}{\deg (v_1)} |a|^2 \bigg) \\
 & \qquad - \frac{1}{D_{v_2}} \bigg( \sum_{j = 1}^{\deg (v_2)} |F_j (v_2)|^2 - \frac{1}{\deg (v_2)} |b|^2 \bigg) \\
 & = \frac{1}{D_{v_0}} \bigg( \frac{1}{\deg (v_1)} |a|^2 + \frac{1}{\deg (v_2)} |b|^2 - \frac{1}{\deg (v_0)} |a + b|^2 \bigg).
\end{align*}
As $\deg (v_0) = \deg (v_1) + \deg (v_2)$ we can apply Lemma~\ref{lem:nabitte} in order to see that the term in the brackets is always nonnegative. From this it follows that $h \leq \widetilde h$ if $D_{v_0} > 0$ and $\widetilde h \leq h$ if $D_{v_0} < 0$. This leads to the assertions (ii) and (iii).
\end{proof}

\end{document}